\documentclass[reqno,a4paper,10pt]{amsart}

\usepackage{graphicx}         
\usepackage{amsmath}
\usepackage{amsfonts}
\usepackage{amssymb}
\usepackage{eucal}
\usepackage[latin1]{inputenc}
\usepackage[all]{xy}

\newtheorem{teo}{Theorem}[section]

\newtheorem{lem}[teo]{Lemma}
\newtheorem{prp}[teo]{Proposition}
\theoremstyle{definition}

\theoremstyle{remark}
\newtheorem{rem}[teo]{Remark}

\renewcommand{\l}{\lambda}
\newcommand{\s}{\sigma}

\newcommand{\PB}{\left\{\cdot\,,\cdot\right\}}

\newcommand{\Pb}[1]{\left\{\cdot\,,#1\right\}}
\newcommand{\pb}[1]{\left\{#1\right\}}

\newcommand{\lb}[1]{\[#1\]}

\renewcommand{\(}{\left(}
\renewcommand{\)}{\right)}
\renewcommand{\[}{\left[}
\renewcommand{\]}{\right]}

\newcommand{\set}[1]{\left\{#1\right\}}

\newcommand{\cB}{\mathcal B}
\newcommand{\cI}{\mathcal I}
\newcommand{\X}{\mathcal X}

\newcommand{\cS}{\mathcal S}

\newcommand{\bbC}{\mathbb C}

\newcommand{\bbN}{\mathbb N}

\newcommand{\bbR}{\mathbb R}

\newcommand{\bbZ}{\mathbb Z}

\newcommand{\bfone}{\mathbf 1}

\newcommand{\um}{{\underline m}}
\newcommand{\un}{{\underline n}}

\newcommand{\cycl}{\circlearrowleft}

\newcommand{\leqs}{\leqslant}
\newcommand{\geqs}{\geqslant}

\newcommand{\Id}{\mathop{\rm Id}}

\renewcommand{\geq}{\geqs}
\renewcommand{\leq}{\leqs}

\newif\ifprivate
\privatefalse

 \numberwithin{equation}{section}

\def\???{\ifprivate {\bf {???}} \marginpar{{\Huge {\bf ?}}}\else \fi}
\numberwithin{equation}{section}

\begin{document}  

\nocite{*}

\parskip 4pt
\baselineskip 16pt


\title[Integrable deformations of some Lotka-Volterra systems]{Integrable deformations of the Bogoyavlenskij-Itoh Lotka-Volterra systems}

\author{C.\,A.~Evripidou}%
\address{
		Charalampos Evripidou,
		  Department of Mathematics and Statistics,
		  La Trobe University, Melbourne,
		  Victoria 3086, Australia}%
\email{C.Evripidou@latrobe.edu.au}

\author{P. Kassotakis}
\address{Pavlos Kassotakis, Department of Mathematics and Statistics,
University of Cyprus, P.O.~Box 20537, 1678 Nicosia, Cyprus}
\email{pavlos1978@gmail.com}

\author{P. Vanhaecke}
\address{ Pol Vanhaecke,  Laboratoire de Math\'ematiques\\
          UMR 7348 du CNRS\\
          Universit\'e de Poitiers\\
          86962 Futuroscope Chasseneuil Cedex\\
          France}
\email{ pol.vanhaecke@math.univ-poitiers.fr}

\thanks{Corresponding author: Pol Vanhaecke,  Email: pol.vanhaecke@math.univ-poitiers.fr}
\date{\today}

\begin{abstract}
  We construct a family of integrable deformations of the Bogoya\-vlenskij-Itoh systems and construct a Lax
  operator with spectral parameter for it. Our approach is based on the construction of a family of compatible
  Poisson structures for the undeformed system, whose Casimirs are shown to yield a generating function for the
  integrals in involution of the deformed systems. We show how these deformations are related to the Veselov-Shabat
  systems.
\end{abstract}


\maketitle

\tableofcontents

\section{Introduction} \label{intro}
The Bogoyavlenskij-Itoh systems were introduced by Bogoyavlenskij in \cite{Bog1} in his study of the integrability of
Lotka-Volterra systems. Recall that the most general form of Lotka-Volterra equations in dimension $n$ is
\begin{equation}\label{eq:LV_gen_intro}
  \dot x_i = \varepsilon_i x_i + \sum_{j=1}^n A_{i,j} x_i x_j, \ \ i=1,2, \dots , n \; .
\end{equation}
For the case of the Bogoyavlenskij-Itoh systems, $n=2k+1$ is odd, there are no linear terms ($\varepsilon_i=0$ for
all $i$) and the matrix $A$ is skew-symmetric with  entries
\begin{equation*}
    A_{i,j}=\left\{
    \begin{array}{rl}
       1\quad&i<j\leqslant\min\set{i+k,2k+1}\;,\\
      -1\quad&\min\set{i+k,2k+1}< j\leqslant 2k+1\;;
    \end{array}
    \right.
\end{equation*}%
see the matrix \eqref{eq:bogo_mat} below. It is a Hamiltonian system, with Poisson structure defined by
\begin{equation}\label{equ:intro_poisson}
  \pb{x_i,x_j}=A_{i,j}x_ix_j\;,
\end{equation}
and linear Hamiltonian $H=x_1+x_2+\dots+x_{2k+1}$. Bogoyavlenskij provides for this
system the following Lax equation (with spectral parameter~$\l$)
\begin{equation}\label{eq:bogo_lax_intro}
  (X+\lambda M)^\cdot=[X+\lambda M,B-\lambda M^{k+1}]
\end{equation}
where for $i,j\in \set{1,2,\dots,2k+1}$ the $(i,j)$-th entry of the matrices $X,\,M$ and $B$ is respectively given
by
\begin{equation*}
  X_{i,j}:=\delta_{i,j+k}x_i\;,\quad M_{i,j}:=\delta_{i+1,j}\;,\quad
  B_{i,j}:=-\delta_{i,j}(x_i+x_{i+1}+\cdots+x_{i+k})\;.
\end{equation*}%
The characteristic polynomial of $X+\lambda M$ leads to $k+1$ independent constants of motion,
among which are both the Hamiltonian $H$ and the Casimir 
\begin{equation}\label{equ:intro_casimir}
  C:=\prod_{i=1}^{2k+1}x_i=x_1x_2\dots x_{2k+1}\;
\end{equation}%
of the Poisson structure \eqref{equ:intro_poisson}. Itoh gives in \cite{itoh1} a combinatorial description of these
integrals and proves in \cite{itoh2} by a combinatorial argument that they are in involution, thereby proving the
Liouville integrability of the Bogoyavlenskij-Itoh system. Some integrable reductions of the Bogoyavlenskij-Itoh
system were recently constructed and studied by us in \cite{PPPP}.

In this paper we construct and study a certain type of integrable deformations of the Bo\-go\-yav\-len\-skij-Itoh
systems. Our approach is based on (compatible) deformations of the Poisson structure \eqref{equ:intro_poisson} and
of its Casimir \eqref{equ:intro_casimir}. We show that the only constant Poisson structures
$$
  \pb{x_i,x_j}_b:=b_{i,j}\;,\qquad 1\leqslant i,j\leqs 2k+1\;,
$$
which are compatible with \eqref{equ:intro_poisson} are the ones for which $b_{i,j}=0$ for all $i$ and $j$ such
that $\vert i-j\vert\notin\set{k,k+1}$. If one writes the corresponding Hamiltonian vector field, with $H$ as
Hamiltonian, one finds
\begin{equation}\label{eq:LV_defo_intro}
  \dot x_i = \sum_{j=1}^{2k+1} A_{i,j} x_i x_j+c_i\;, \ \ i=1,2, \dots , 2k+1 \; ,
\end{equation}
where $c_i=b_{i,i+k}-b_{i-k,i}$. Notice that the $c_i$ sum up to zero, but this is the only relation between these
constants. Up to a minor change of variables (see Section \ref{sec:veselov}) this system coincides with the
so-called \emph{Veselov-Shabat system}, which was constructed in \cite{VS} as fixed point of compositions of
Darboux transformations of the Schr\"{o}dinger operator. Notice that this system is also known as (one of) the
\emph{Noumi-Yamada} system(s) (see \cite{NY}), but it seems that Noumi and Yamada were unaware of the
Veselov-Shabat system. Also, neither in the Noumi-Yamada paper \cite{NY} nor in the Veselov-Shabat paper \cite{VS}
are the systems that they consider put in relation with the Bogoyavlenskij-Itoh systems, of which they are
deformations.

We prove the Liouville integrability of \eqref{eq:LV_defo_intro} by using the deformed Casimir, i.e., the Casimir
of the deformed Poisson structure $\PB+\PB_b$. A combinatorial description of the deformed Casimir is given in
Proposition \ref{prop:deformed_casimir}. From it, one gets immediately a Casimir for the Poisson pencil
$\PB+\l\PB_b$ and so, using the Lenard-Magri scheme, a family of polynomials in involution, which we show to be
independent, thereby proving the Liouville integrability of the deformed Bogoyavlenskij-Itoh systems. We also
provide a Lax equation with spectral parameter for these systems, which is a deformation of the Lax equation
\eqref{eq:bogo_lax_intro}. Notice that our $(2k+1)\times(2k+1)$ Lax operator is different from the $2\times 2$
Lax operator which was constructed by Veselov and Shabat; our Lax equation has the advantage that the phase
variables appear in it linearly, which makes it easier to extract the phase variables (in order to know for example
their time evolution) from the Lax operator.

It can be shown that the deformed Bogoyavlenskij-Itoh systems admit a natural discretization, which is constructed
by using the deformed Lax operator which we constructed in this paper. This discretization will be worked out and
studied in a future publication.

The structure of the paper is as follows. We recall the main facts about the Bogoyavlenskij-Itoh systems in Section
\ref{sec:bogo}. The combinatorial constructions that we will use, in the style of Itoh's combinatorial
constructions of the (undeformed) first integrals, are prepared in Section \ref{sec:ell}. We proceed in Section
\ref{sec:poisson_defo} with the construction of the deformed Poisson strutures and deformed Casimirs. Section
\ref{sec:Lax} is devoted to the Lax equation and the Liouville integrability of the deformed Bogoyavlenskij-Itoh
systems; in particular we match the combinatorial description of the first integrals with the coefficients of our
Lax equation. The upshot is that \emph{as integrable systems} they are deformation of the Bogoyavlenskij-Itoh
systems. In the last section we show how these deformed systems are related to the ones constructed by Veselov and
Shabat in \cite{VS}.


\section{The Bogoyavlenskij-Itoh systems and their integrability}\label{sec:bogo}
We recall in this section the basic results on the Bogoyavlenskij-Itoh systems, which were obtained in
\cite{Bog1,Bog2,itoh1,itoh2}; the notation is a slight simplification of the one used in \cite{PPPP}, where more
general systems are considered. Fix $k\in\bbN^*$ and denote $\cI:=\set{1,2,\dots,2k+1}$. We consider on
$\bbR^{2k+1}$ (or on $\bbC^{2k+1}$) with linear coordinates $x_1,\dots,x_{2k+1}$ the homogeneous quadratic Poisson
structure defined by
\begin{equation}\label{eq:quad_poisson}
  \pb{x_i,x_j}:=A_{i,j}x_ix_j\;,\qquad i,j\in \cI\;,
\end{equation}
where $A=(A_{i,j})$ is the constant skew-symmetric matrix
\begin{equation}\label{eq:bogo_mat}
  A=
  \begin{pmatrix}
    0&1&1&\cdots&1&-1&-1&\cdots&-1&-1\\
   -1&0&1&\cdots&1&1&-1&\cdots&-1&-1\\
   -1&-1&0&\cdots&1&1&1&\ddots&-1&-1\\
   \vdots&\vdots& & \ddots & \vdots & \vdots & \vdots & \ddots & \ddots & \vdots \\
   -1&-1&-1&\cdots&\cdots&\cdots&\cdots&\cdots&1&-1\\
    1&-1&-1&\cdots&\cdots&\cdots&\cdots&\cdots&1&1\\
    \vdots&\vdots& & \ddots & \vdots & \vdots & \vdots & \ddots & \vdots & \vdots  \\
    1&1&1&\cdots&-1&-1&-1&\cdots&0&1\\
    1&1&1&\cdots&1&-1&-1&\cdots&1&0
  \end{pmatrix},
\end{equation}
with $1$ and $-1$ appearing $k$ times on the first row (and hence on every row). When $k$ is not clear from the context, we
write $A^{(k)}$ for $A$ and $\PB^{(k)}$ for $\PB$. The rank of $A$ is constant and equals $2k$; a Casimir for $\PB$
is given by the polynomial function
\begin{equation*}
  C:=\prod_{i\in \cI}x_i=x_1x_2\dots x_{2k+1}\;.
\end{equation*}%
As Hamiltonian function, one takes $H:=\sum_{i\in \cI}x_i=x_1+x_2+\dots+x_{2k+1}$, the sum of all
coordinates. If we set $x_{2k+\ell+1}=x_\ell$ for all $\ell\in\bbZ$, then the Hamiltonian vector field
$\X_H:=\Pb{H}$ is given by
\begin{equation}\label{eq:BI_prelim}
  \dot x_i=x_i\sum_{j=1}^k(x_{i+j}-x_{i-j})\;,\qquad i\in \cI\;.
\end{equation}%
The automorphism of $\bbR^{2k+1}$ of order $2k+1$, given by
\begin{equation}\label{eq:auto_prelim}
  (x_1,x_2,\dots,x_{2k},x_{2k+1})\mapsto (x_2,x_3,\dots,x_{2k+1},x_1)\;,
\end{equation}%
is a Poisson map and it preserves the Hamiltonian, hence it is an automorphism of the system.  
The following Lax equation (with spectral parameter~$\l$) was provided by Bogoyavlenskij in
\cite{Bog1}:
\begin{equation}\label{eq:bogo_lax}
  (X+\lambda M)^\cdot=[X+\lambda M,B-\lambda M^{k+1}]
\end{equation}
where for $i,j\in \cI$ the $(i,j)$-th entry of the matrices $X,\,M$ and $B$ is respectively given by
\begin{equation}\label{eq:bogo_lax_entries}
  X_{i,j}:=\delta_{i,j+k}x_i\;,\quad M_{i,j}:=\delta_{i+1,j}\;,\quad
  B_{i,j}:=b_i\delta_{i,j}:=-\delta_{i,j}(x_i+x_{i+1}+\cdots+x_{i+k})\;.
\end{equation}%
%
The characteristic polynomial of $X+\lambda M$ has the form
\begin{equation}\label{eq:bogo_char_poly}
  \det(X+\lambda M-\mu\Id)=\l^{2k+1}-\mu^{2k+1}+\sum_{\ell=0}^kK_\ell\l^{k-\ell}\mu^{k-\ell}\;,
\end{equation}%
where, by homogeneity, each $K_\ell$ is a homogeneous polynomial (in $x_1,\dots,x_{2k+1}$) of degree $2\ell+1$. One
has $K_0=H$, the Hamiltonian, and $K_k=C$, the above Casimir. Being a coefficient of the characteristic polynomial
of the Lax operator $X+\l M$, each one of the $K_\ell$ is a first integral of \eqref{eq:BI_prelim}.

In order to give an explicit formula for $K_\ell$, we need some more notation, which will also be needed in the
rest of the paper. With $k$ fixed as above, let $\ell$ be an integer with $0\leqs \ell\leqs k$. We will consider
besides $A=A^{(k)}$ also $A^{(\ell)}$. Let $\um=(m_1,m_2,\dots,m_{2\ell+1})$ be a $(2\ell+1)$-tuple of integers,
satisfying $1\leqs m_1<m_2<\cdots<m_{2\ell+1}\leqs 2k+1$. We view them as indices of the rows and columns of $A$:
we denote by $A'_{\um}$ the square submatrix of $A$ of size $2\ell+1$, corresponding to rows and columns
$m_1,m_2,\dots,m_{2\ell+1}$ of~$A$, so that
\begin{equation}\label{eq:B_def}
  (A'_{\um})_{i,j}=A_{m_i,m_j}\;, \quad\hbox{ for } i,j=1,\dots,2\ell+1\;.
\end{equation}
Let
\begin{equation}\label{eq:S_def}
  \cS_{\ell}:=\set{\um\mid A'_{\um}=A^{(\ell)}}\;.
\end{equation}
With this notation, the polynomials $K_\ell$ which appear in the characteristic polynomial
\eqref{eq:bogo_char_poly} can be written as
\begin{equation}\label{eq:k_i_itoh}
  K_\ell=\sum_{\um\in \cS_{\ell}} x_{m_1}x_{m_2}\dots x_{m_\ell}\dots x_{m_{2\ell+1}}\;.
\end{equation}%
For example, $\cS_{0}=\set{1,2,\dots,2k+1}$ and $\cS_{k}=\set{(1,2,\dots,2k+1)}$, so that $K_0=H$ and $K_k=C$, as
above. Moreover, Itoh shows by a beautiful combinatorial argument that the polynomials $K_\ell$ are in involution,
$\pb{K_\ell,K_m}=0$ for $0\leqs \ell<m \leqs k$. Since these $k+1$ polynomials are moreover functionally
independent, and since the rank of the Poisson structure $\PB$ is $2k$, the triplet
$(\bbR^{2k+1},\PB,(K_0,K_1,\dots,K_k))$ is a Liouville integrable system.


\section{The sets $\cS_{\ell}$}\label{sec:ell}
We establish in this section some combinatorial properties of the sets $\cS_{\ell}$ which will be used in the
subsequent sections. We recall that $k$ is a fixed integer and that $\cI$ stands for $\set{1,2,\dots,2k+1}$. We
also denote for $s \in\cI$ by $\cI^{(s)}$ the set of all strictly ordered $s$-tuplets $\um=(m_1,m_2,\dots,m_s)$,
with entries in $\cI$. We will often use set theory notation for such elements, for example we write $i\in\um$ and
$\set{i,j}\subset\um$, with the obvious meanings.  If $\um$ and $\un$ are two vectors with elements in $\cI$,
satisfying $\um\cap\un=\emptyset$, we write $\um\oplus\un$ for the vector which contains the elements of $\um$ and
$\un$ (in increasing order); also, if $\un\subset\um$, we write $\um\ominus\un$ for the vector with elements in
$\um$ and not in $\un$.  By a slight abuse of notation, we will denote for $r\in\bbZ$ by $r\mod 2k+1$ the unique
element of~$\cI$ which is congruent to $r$ modulo $2k+1$. For $0\leqs \ell\leqs k$ the set $\cS_{\ell}$ is given
by
\begin{equation*}
  \cS_{\ell}=\set{\um\in\cI^{(2\ell+1)}\mid A_{m_i,m_j}=A^{(\ell)}_{i,j}\;,\hbox{ for } i,j=1,\dots,2\ell+1}\;.
\end{equation*}
We recall from \cite{PPPP} the following characterization of the elements of $\cS_{\ell}$:
\begin{prp}\label{lma:S}
  Let $\um=(m_1,\dots,m_{2\ell+1})$ be an element of $\cI^{(2\ell+1)}$. Then $\um \in \cS_{\ell}$ if and only if
  the following conditions are satisfied:
  \begin{enumerate}
    \item[(1)] $m_{\ell+i}<m_i+k+1\leqs m_{\ell+i+1}$ for $i=1,\dots,\ell$;
    \item[(2)] $m_{2\ell+1}<m_{\ell+1}+k+1$.
  \end{enumerate}
\end{prp}

There are for any $s\in\cI$ two natural permutations of $\cI^{(s)}$: first, there is an involution $\s_s$, given by
\begin{equation*}
  \s_s(m_1,m_2,\dots,m_{s}):=(2k+2-m_{s},2k+2-m_{s-1},\dots,2k+2-m_{1})\;.
\end{equation*}%
If we write $\un:=\s_s(\um)$, then $n_i=2k+2-m_{s+1-i}$, for $i=1,2,\dots,s$. Next, there is a cyclic
permutation of order $2k+1$,
\begin{equation}\label{eq:tau_def}
  \tau_s(m_1,m_2,\dots,m_{s}):=
  \left\{\begin{array}{ll}
    (m_1+1,m_2+1,\dots,m_{s}+1)&\hbox{when } m_{s}<2k+1\;;\\
    (1,m_1+1,m_2+1,\dots,m_{s-1}+1)&\hbox{when } m_{s}=2k+1\;.
  \end{array}
  \right.
\end{equation}%
Said differently, $\tau_{s}$ simply adds 1 to all entries of the vector $\um$, but the result needs to be slightly
reordered when one of the entries of $\um$ gets bigger than $2k+1$. We show in following lemma that $\s_{2\ell+1}$
and $\tau_{2\ell+1}$ both restrict to a permutation of $\cS_{\ell}$.
\begin{lem}\label{lma:tau}
  If $\um\in\cS_{\ell}$ then $\s_{2\ell+1}(\um)\in\cS_{\ell}$ and $\tau_{2\ell+1}(\um)\in\cS_{\ell}$.  
\end{lem}
\begin{proof}
Suppose that $\um\in\cS_{\ell}$ and let $\un:=\s_{2\ell+1}(\um)$. According to Proposition \ref{lma:S} we need to show that
\begin{equation}\label{eq:tau_ineqs}
  1\leqs i\leqs \ell+1\implies n_{\ell+i}<n_i+k+1\hbox{ and } 1\leqs i\leqs\ell\implies n_i+k+1\leqs n_{\ell+i+1}\;.
\end{equation}
In terms of $\um$ these two inequalities become $m_{2\ell+2-i}<m_{\ell+2-i}+k+1$ and $m_{\ell+1-i}+k+1\leqs
m_{2\ell+2-i}$.  Setting $j:=\ell+2-i$ in the first inequality and $j:=\ell+1-i$ in the second inequality,
\eqref{eq:tau_ineqs} becomes
\begin{equation*}
  1\leqs j\leqs \ell+1\implies m_{\ell+j}<m_j+k+1\hbox{ and } 1\leqs j\leqs\ell\implies m_j+k+1\leqs m_{\ell+j+1}\;,
\end{equation*}%
which are exactly the conditions (1) and (2) in Proposition \ref{lma:S} which express that $\um\in\cS_{\ell}$. This
shows that $\s_{2\ell+1}(\um)\in\cS_{\ell}$. Next, let $\un:=\tau_{2\ell+1}(\um)$. Again, we need to verify
\eqref{eq:tau_ineqs}: when $m_{2\ell+1}<2k+1$ this is completely obvious, so let us assume that
$m_{2\ell+1}=2k+1$. We have that $n_1=1$ and $n_i=m_{i-1}+1$ for $i=2,3,\dots,2\ell+1$. We need to prove
\eqref{eq:tau_ineqs} for this vector $\un$.  Again, for $i\geqs2$ this is completely obvious, so we only need to
check that $m_\ell<k+1\leqs m_{\ell+1}$. Both follow from the characterizations of $\um$ in Proposition
\ref{lma:S}: the first one from the second inequality in (1), with $i=\ell$, and the second one from~(2), with
$m_{2\ell+1}=2k+1$.
\end{proof}

We decompose $\cS_{\ell}$ in two subsets $\cS_{\ell,+}$ and $\cS_{\ell,-}$, where
\begin{equation}%
\label{eq:decomp_of_S_i}
  \cS_{\ell,+}:=\set{\um\in\cS_{\ell}\mid 1\in\um}\;,\quad   \cS_{\ell,-}:=\set{\um\in\cS_{\ell}\mid 1\notin\um}\;
\end{equation}%
and we define the maps
\begin{equation}\label{eq:phi}
  \begin{array}{cclclccl}
    \phi_1&:&\set{\um\in\cS_{\ell-1,-}\mid m_\ell\leq k+1}&\rightarrow&\cS_{\ell},&\um&\mapsto&\um\oplus(1,k+2)\;,\\
    \phi_2&:&\set{\um\in\cS_{\ell-1,-}\mid m_\ell\geq k+2}&\rightarrow&\cS_{\ell},&\um&\mapsto&\um\oplus(1,k+1)\;.
  \end{array}  
\end{equation}
\begin{lem}\label{lma:phi_maps}
  The formulas \eqref{eq:phi} define injective maps with values in $\cS_{\ell,+}$. More precisely, they are part of
  the following commutative diagram, where all (restriced) maps are bijections.
  \begin{equation*}\label{eq:phi_diag}
  \xymatrix{
  \set{\um\in\cS_{\ell-1,-}\mid m_\ell\leq k+1}
  \ar[rr]^-{\phi_1}
  \ar[d]_-{\tau_{\ell-1}\circ\s_{\ell-1}}& &
  \set{\un\in\cS_{\ell,+}\mid k+2\in\un}\ar[d]^-{\tau_\ell\circ\s_{\ell}}\\
  \set{\um\in\cS_{\ell-1,-}\mid m_\ell\geq k+2}
  \ar[rr]^-{\phi_2}& &
  \set{\un\in\cS_{\ell,+}\mid k+1\in\un}}
  \end{equation*}
\end{lem}
\begin{proof}
We first show that $\phi_1$ is well-defined. Let $\ell\geqs1$ and suppose that $\um\in\cS_{\ell-1,-}$ with
$m_\ell\leq k+1$. We need to show that $\um\cap(1,k+2)=\emptyset$ and that $\un:=\phi_1(\um)=\um\oplus(1,k+2)$
belongs to $\cS_{\ell}$. According to Proposition \ref{lma:S},
\begin{eqnarray}
  &&m_{\ell-1+i}<m_i+k+1\leqs m_{\ell+i} \quad\hbox{ for }\quad i=1,\dots,\ell-1\;;\label{eq:ms}\\
  &&m_{2\ell-1}<m_{\ell}+k+1\;,\label{eq:mi}
\end{eqnarray}
and, by assumption,
\begin{equation}\label{eq:assum}
  m_1>1\qquad\hbox{and} \qquad m_\ell\leqs k+1\;.
\end{equation}%
According to \eqref{eq:ms} with $s=1$ and \eqref{eq:assum}, $m_{\ell+1}>k+2>m_\ell$, which shows that
$\um\cap(1,k+2)=\emptyset$. Since $m_{\ell+1}>k+2$ and $m_\ell\leqs k+1$,
\begin{equation}\label{eq:m_to_n}
  n_j=\left\{
  \begin{array}{cl}
    1&j=1\;,\\
    m_{j-1}&j=2,\dots,\ell+1\;,\\
    k+2&j=\ell+2\;,\\
    m_{j-2}&j=\ell+3,\dots,2\ell+1\;.
  \end{array}
  \right.
\end{equation}%
We need to check that $\un\in\cS_{\ell}$, i.e., that
\begin{eqnarray}
  &&n_{\ell+i}<n_i+k+1\leqs n_{\ell+i+1} \quad\hbox{ for }\quad i=1,\dots,\ell\;;\label{eq:ns}\\
  &&n_{2\ell+1}<n_{\ell+1}+k+1\;.\label{eq:ni}
\end{eqnarray}
To do this, we use \eqref{eq:m_to_n} to write the latter inequalities in terms of the entries of $\um$: for $i=1$,
$i=2$ and $i=3,\dots,\ell$, \eqref{eq:ns} reads respectively
\begin{eqnarray*}
  &&m_\ell<k+2\leqs k+2\;,\\
  &&k+2<m_1+k+1\leqs m_{\ell+1}\;,\\
  &&m_{\ell+i-2}<m_{i-1}+k+1\leqs m_{\ell+i-1}\;.
\end{eqnarray*}
The first line follows from $m_\ell\leqs k+1$ (see \eqref{eq:assum}), the other two from $m_1>1$
and~\eqref{eq:ms}. Finally, \eqref{eq:ni}, written in terms of $\um$ is precisely \eqref{eq:mi}. This shows that
$\un\in\cS_{\ell}$ and hence that $\phi_1$ takes values in $\cS_{\ell,+}$, more precisely in the set
$\left\{\un\in\cS_{\ell,+}\mid(1,k+2)\in\un\right\}$. Obviously, $\phi_1$ is injective. It is proven similarly that
$\phi_2$ takes values in $\cS_{\ell,+}$ and is injective; let us just point out that the formula for
$\un:=\phi_2(\um)$ (as in \eqref{eq:m_to_n}) is now given by
\begin{equation}\label{eq:m_to_n_2}
  n_j=\left\{
  \begin{array}{cl}
    1&j=1\;,\\
    m_{j-1}&j=2,\dots,\ell\;,\\
    k+1&j=\ell+1\;,\\
    m_{j-2}&j=\ell+2,\dots,2\ell+1\;.
  \end{array}
  \right.
\end{equation}%
In order to show that the (restricted) horizontal maps $\phi_1$ and $\phi_2$ in the diagram are bijections, it
suffices to construct their inverse maps, which are defined as restrictions of the following two maps:
\begin{equation*}
  \begin{array}{cclclccl}
    \phi_1^{-1}&:&\set{\um\in\cS_{\ell,+}\mid k+2\in\um}&\rightarrow&\cS_{\ell},&\um&\mapsto&\um\ominus(1,k+2)\;,\\
    \phi_2^{-1}&:&\set{\um\in\cS_{\ell,+}\mid k+1\in\um}&\rightarrow&\cS_{\ell},&\um&\mapsto&\um\ominus(1,k+1)\;.
  \end{array}  
\end{equation*}
Clearly, these maps are well-defined. We need to show that if $\un\in\cS_{\ell}$ with $(1,k+2)\subset\un$ then
$\um:=\un\ominus(1,k+2)$ satisfies $m_\ell\leqs k+1$. Since $\un\in\cS_{\ell}$ it satisfies \eqref{eq:ns}, which
yields with $i=1$ and $n_1=1$ that $n_{\ell+1}<k+2\leqs~n_{\ell+2}$ so that $n_{\ell+2}=k+2$ and
$m_\ell=n_{\ell+1}\leqs k+1$.  Similarly, if $\un\in\cS_{\ell}$ with $(1,k+1)\subset\un$ then
$\um:=\un\ominus(1,k+1)$ satisfies $m_\ell>k+1$. This shows that the horizontal arrows in the diagram are
bijections.

Let us show that the vertical arrows of the diagram are also bijections. In view of Lemma \ref{lma:phi_maps},
$\tau_{\ell-1}\circ\s_{\ell-1}$ is a permutation (involution) of $\cS_{\ell-1}$; explicitly, it is given for
$\um\in\cS_{\ell-1,-}$ by
\begin{equation}
  \tau_{\ell-1}\circ\s_{\ell-1}(\um)=(2k+3-m_{2\ell-1},2k+3-m_{2\ell-2},\dots,2k+3-m_{1})\;.
\end{equation}
The formula shows that $1\notin \tau_{\ell-1}\circ\s_{\ell-1}(\um)$ so that $\tau_{\ell-1}\circ\s_{\ell-1}$
restricts to a bijection of $\cS_{\ell-1,-}$, and that the $\ell$-th component of
$\tau_{\ell-1}\circ\s_{\ell-1}(\um)$ is given by $2k+3-m_\ell$, showing that $\tau_{\ell-1}\circ\s_{\ell-1}$
restricts further to a bijection between the subsets $\set{\um\in\cS_{\ell-1,-}\mid m_\ell\leq k+1}$ and
$\set{\um\in\cS_{\ell-1,-}\mid m_\ell\geq k+2}$ of $\cS_{\ell,-}$. Similarly, $\tau_{\ell}\circ\s_{\ell}$ is an
involution of $\cS_{\ell}$, which is given, for $\un\in\cS_{\ell,+}$ by
\begin{equation}\label{eq:tau_sigma}
  \tau_{\ell}\circ\s_{\ell}(\un)=(1,2k+3-n_{2\ell+1},2k+3-n_{2i},\dots,2k+3-n_{3},2k+3-n_{2})\;.
\end{equation}
Also, it restricts to a bijection between the subsets $\set{\un\in\cS_{\ell,+}\mid (1,k+2)\subset\un}$ and
$\left \{\un\in\cS_{\ell,+}\mid \right.$ $\left. (1,k+1)\subset\un\right\}$ of $\cS_{\ell,+}$: if $(1,k+2)\subset\un$
(resp.\ $(1,k+1)\subset\un$), then $n_{\ell+2}=k+2$ (resp.\ $n_{\ell+1}=k+1)$, so that $\tau_\ell(\s_\ell(\un))\ni
k+1$ (resp.\ $\tau_\ell(\s_i(\un))\ni k+2$). The commutativity of the diagram follows at once from the explicit
formulas \eqref{eq:m_to_n} and \eqref{eq:m_to_n_2}~--~\eqref{eq:tau_sigma}.
\end{proof}

For an element $\um\in\cI^{(s)}$, we denote by $\um'$ the vector whose elements are those elements of $\cI$
which are absent from $\um$ (again these elements will always be put in the increasing order). In formula,
$\um'=\cI\ominus\um$. We also denote, for $0\leqs\ell\leqs k$,
\begin{equation*}
  \cS'_\ell:=\set{\um'\mid\um\in\cS_{\ell}}\;.
\end{equation*}%
Clearly, if $\um'\in\cS_{\ell}'$ then $\um'$ has $2(k-\ell)$ entries. We give in the following proposition a
characterization of the elements of $\cS'_\ell$; it will be used in the construction of first integrals in the following
sections.
\begin{prp}\label{prp:m'}
  Let $0\leqs\ell\leqs k$ and let
  $\um'=(r_1,r_2,\dots,r_{k-{\ell}},s_1,s_2,\dots,s_{k-{\ell}})\in\cI^{(2k-2\ell)}$. Then $\um'\in\cS_{\ell}'$ if
  and only if
  \begin{equation*}
    s_j-r_j\in\set{k,k+1}\;,\quad\hbox{for }j=1,\dots,k-{\ell}\;.
  \end{equation*}%
\end{prp}
\begin{proof}
The proof goes by induction on $\ell$, starting from $\ell=k$, downwards to $\ell=0$. When $\ell=k$, the
equivalence is trivially satisfied, because $\cS_k$ has a single element $\um=(1,2,\dots,2k+1)$, whose complement
$\um'$ has no entries. Suppose now that the above equivalence is true for some $\ell$ with $0<\ell\leqs k$ and for
all $\um'\in\cI^{(2k-2\ell)}$. Let $\um'=(r_1,r_2,\dots,r_{k-{\ell}+1},s_1,s_2,\dots,r_{k-{\ell}+1})$. We prove
that
\begin{equation}\label{eq:s-r}
  \um'\in\cS_{\ell-1}' \quad\mbox{if and only if} \quad s_j-r_j\in\set{k,k+1}\;,\quad\hbox{for
  }j=1,\dots,k-{\ell}+1\;.
\end{equation}
Let us assume first that $\um'\in\cS_{\ell-1}'$, i.e., that $\um\in\cS_{\ell-1}$. According to Lemma \ref{lma:tau},
$\tau_{2\ell-1}(\um)\in\cS_{\ell-1}$ and it is easy to see that if $\um'$ satisfies the right hand side of
\eqref{eq:s-r} then so does $\tau_{2k-2\ell+2}(\um')=\tau_{2\ell-1}(\um)'$.  We may therefore assume that
$1\notin\um$, but also that $m_\ell\leqs k+1$: indeed, if $1\notin\um$ and $m_\ell>k+1$ then it suffices to replace
$\um$ by $\tau_{2\ell-1}^k(\um)$ (note that $m_{\ell-1}\leqs k$, as follows easily by taking $s=\ell-1$ in item (1)
of Proposition~\ref{lma:S}).  According to Lemma \ref{lma:phi_maps}, $\um\oplus(1,k+2)\in\cS_{\ell}$. By the
recurrence hypothesis, we can write $(\um\oplus(1,k+2))'=(\bar r_2,\dots_,\bar r_{k-{\ell}+1},\bar s_2,\dots,\bar
s_{k-{\ell}+1})$, satisfying $s_j-r_j\in\set{k,k+1}$ for $j=2,\dots,k-{\ell}+1;$ also, $\bar r_2>1$ and $\bar
r_{k-{\ell}+1}<k+2<\bar s_2$. Then
\begin{equation*}
  (1,\bar r_2,\dots_,\bar r_{k-{\ell}+1},k+2,\bar s_2,\dots,\bar
  s_{k-{\ell}+1})=\um'=(r_1,r_2,\dots,r_{k-{\ell}},s_1,s_2,\dots,s_{k-{\ell}})\;,
\end{equation*}%
so that $r_j=\bar r_j$ and $s_j=\bar s_j$ for $j=2,3,\dots,k_\ell+1$. As a consequence, the entries of $\um'$
satisfy the right hand side of \eqref{eq:s-r}.

We now assume that the entries of $\um'$ satisfy the right hand side of \eqref{eq:s-r}. Again, using
$\tau_{2\ell-1}$ we may assume that $r_1=1$ and that $s_1=k+2$. Thus, $1\notin\um$ and $k+2\notin\um$. Then
\begin{equation*}%
  (\um\oplus(1,k+2))'=(r_2,r_3,\dots,r_{k-{\ell}+1},s_2,s_3,\dots,s_{k-{\ell}+1}) 
\end{equation*}%
with $s_j-r_j\in\set{k,k+1}$ for $j=2,\dots,k-{\ell}+1$. By the recursion hypothesis,
$\um\oplus(1,k+2)\in\cS_{\ell}$; more precisely, $\um\oplus(1,k+2)\in\set{\un\in\cS_{\ell,+}\mid k+2\in\un}$. By
Lemma \ref{lma:phi_maps}, $\um\in\cS_{\ell-1}$, and so $\um'\in\cS_{\ell-1}'$, as was to be shown.
\end{proof}


\section{The deformed Poisson structure and its basic Casimir}\label{sec:poisson_defo}
In this section, we introduce a class of deformations of the Poisson structure~$\PB$, defined by
\eqref{eq:quad_poisson} and construct a Casimir for it. In the following proposition, we determine all constant
Poisson structures on $\bbR^{2k+1}$ which are compatible with $\PB$.
\begin{prp}\label{prop:compatible}
  Let $(b_{i,j})_{1\leqs i,j\leqs 2k+1}$ be an arbitrary skew-symmetric matrix with entries in $\bbR$ and consider
  the corresponding constant Poisson structure on $\bbR^{2k+1}$, defined by $\pb{x_i,x_j}_b:=b_{i,j}$ for $1\leqs
  i,j\leqs 2k+1$. Then $\PB_b$ is compatible with the quadratic Poisson structure $\PB$ if and only if $b_{i,j}=0$
  for all $i$ and $j$ such that $\vert i-j\vert\notin\set{k,k+1}$. In particular, the constant Poisson structures
  which are compatible with $\PB$ form a vector space of dimension $2k+1$.
\end{prp}
\begin{proof}
Recall (e.g.\ from \cite[Section 3.3.2]{PLV}) that two Poisson structures are said to be \emph{compatible} when
their sum is a Poisson structure, i.e., when it satisfies the Jacobi identity. Denote
\begin{equation}\label{eq:PB_sum}
  \PB_b^{(k)}:=\PB+\PB_b\;.
\end{equation}
Then $\PB_b^{(k)}$ satisfies the Jacobi identity if and only if
\begin{equation*}
  \pb{\pb{x_i,x_j}_b^{(k)},x_m}_b^{(k)}+\cycl (i,j,m)=0\;,
\end{equation*}%
for all triplets of distinct indices $i,j,m\in\cI$, where $\cycl (i,j,m)$ has the obvious meaning.  Since
$\PB_b$ is a constant Poisson bracket this condition amounts to
%
%
$  \pb{\pb{x_i,x_j},x_m}_b+\cycl (i,j,m)=0\;,$
that is,
\begin{gather*}
  A_{i,j}\pb{x_ix_j,x_m}_b+\cycl (i,j,m)=
  A_{i,j}x_ib_{j,m}+A_{i,j}x_jb_{i,m}+\cycl (i,j,m)=\\
  A_{i,j}x_ib_{j,m}+A_{m,i}x_ib_{m,j}+\cycl (i,j,m)=
  x_ib_{j,m}\left(A_{i,j}+A_{i,m}\right)+\cycl (i,j,m)=0\;,
\end{gather*}%
which is in turn equivalent with
\begin{equation}\label{eq:b_times_sum_of_A}
  b_{j,m}\left(A_{i,j}+A_{i,m}\right)=0
\end{equation}
for all triplets of distinct indices $i,j,m\in\cI$. It follows that $\PB$ and $\PB_b$ are compatible if and only
if $b_{j,m}=0$ for all $j,m\in\cI$ with $j\neq m$ such that $A_{i,j}+A_{i,m}\neq 0$ for some $i\in\cI$, with $i\neq j$
and $i\neq m$.

Suppose that $j$ and $m$ satisfy $\vert j-m\vert\in\set{k,k+1}$; without loss of generality, we may assume
that $m=j+k$. From the structure of the $i$-th line of the matrix $A$ (see \eqref{eq:bogo_mat}) we see that
$A_{i,j}+A_{i,j+k}=0$ whenever $i\not\in\set{ j, j+k}$. Therefore, there are no constraints on the entries
$b_{j,m}$, when $\vert j-m\vert\in\set{k,k+1}$. We show that the other entries $b_{j,m}$ all have to be
zero. By permuting $j$ and $m$ if necessary, we may suppose that $j<m$. If $m-j<k$, then picking $i=j-1$
(or $i=m+1$ in case $j=1$) we have $A_{i,j}+A_{i,m}\neq0$; one arrives at the same result when $m-j>k+1$ by
picking $i=j+1$.  In both cases, it follows that the corresponding coefficient $b_{j,m}$ must be zero.
\end{proof}

Notice that the conditions on $b$ which are given in the above proposition are void when $k=1$.
\begin{rem}
In what follows we fix a constant skew-symmetric matrix $b$ with
$b=(b_{i,j})_{1\leqs i,j\leqs 2k+1}$ and $b_{i,j}=0$ for all $i$ and $j$ such that $\vert i-j\vert\notin\set{k,k+1}$.
We also consider the Poisson structure $\PB_b^{(k)}:=\PB+\PB_b$ as in the proof of the previous proposition.
Note that for any other constant skew-symmetric matrix $(c_{i,j})_{1\leqs i,j\leqs 2k+1}$
with $c_{i,j}=0$ for all $i$ and $j$ such that $\vert i-j\vert\notin\set{k,k+1}$
the Poisson structures $\PB_b^{(k)}$ and $\PB_c^{(k)}$ are compatible.
\end{rem}

The rank of the Poisson structure $\PB$ is $2k$, as follows from the fact that the rank of $A$ is $2k$ (see
e.g.\ \cite[Example~8.14]{PLV}). Thus, the Poisson matrix of $\PB$ has a non-vanishing $(2k\times 2k)$-minor; the
latter is a homogeneous polynomial $P_{4k}(x)$ of degree $4k$ since the entries of the Poisson matrix are
homogeneous and quadratic. The corresponding $(2k\times 2k)$-minor of the Poisson matrix of $\PB_b^{(k)}$ is a
non-homogeneous polynomial of degree $4k$, whose homogeneous part of top degree is $P_{4k}(x)$, hence this minor is
also non-zero (at a generic point of $\bbR^{2k+1}$), so that the rank of $\PB_b^{(k)}$ is $2k$.

We show that $\PB_b^{(k)}$ has a polynomial Casimir of degree $2k+1$ whose leading term is the Casimir
$C=x_1x_2\dots x_{2k+1}$ of $\PB$. To do this, we need some more notation. First, in order to eliminate many signs
in the formulas that follow, we define new constants $b'_{i,j}$, which are just the constants $b_{i,j}$, possibly
up to a sign: we define, $b'_{i,j}:=-b_{i,j}=b_{j,i}$ when $\vert i-j\vert=k+1$ and $b'_{i,j}:=b_{i,j}$ otherwise
(i.e., when $|i-j|=k$). Also, for $\um\in\cS_\ell$, we denote by $x_\um$ the product
$x_{\um}=x_{m_1}x_{m_2}\dots x_{m_{2\ell+1}}$ and by $b'_{\um}$ the product
\begin{equation*}
  b'_{\um}:=b'_{r_1,s_1}b'_{r_2,s_2}\dots b'_{r_{k-\ell},s_{k-\ell}}\;,
\end{equation*}%
where $\um'=(r_1,\dots,r_{k-\ell},s_1,\dots,s_{k-\ell})$ is the ordered list of all indices of $\cI$ that are
absent in $\um$ (see Proposition \ref{prp:m'}).

\begin{prp}\label{prop:deformed_casimir}
  The polynomial
  \begin{equation}\label{eq:defo_cas}
    K^b_k:=\sum_{\ell=0}^k\sum_{\um\in\cS_\ell}b'_\um x_\um
  \end{equation}%
  is a Casimir of $\PB_b^{(k)}=\PB+\PB_b$.
\end{prp}
\begin{proof}
In order to simplify the notation in the proof, we will write $K^b_k$ simply as $C$.  We need to prove that
$\pb{x_j,C}_b^{(k)}=0$ for all $j\in\cI$. To do this, it is sufficient to prove that
$\pb{x_1,C}_b^{(k)}=0$. Indeed, consider the cyclic permutation $\tau:\cI\to\cI$ which sends $j$ to $j+1$, for
$j<2k+1$ and sends $2k+1$ to $1$ (it is $\tau_1$, as defined in \eqref{eq:tau_def}). The induced permutation on the
variables $x_i$ is denoted by $\tau^*$, so that $\tau^*x_1=x_2$, for example. As we already recalled, $\tau^*$ is a
Poisson automorphism of $\PB$; it is also a Poisson isomorphism between $\PB_b$ and $\PB_{\tau^*b}$, where
$\tau^*b$ is defined by $(\tau^*b)_{i,j}:=b_{i+1,j+1}$. Indeed,
\begin{equation*}
  \pb{\tau^*x_i,\tau^*x_j}_b=\pb{x_{i+1},x_{j+1}}_b=b_{i+1,j+1}=\pb{x_i,x_j}_{\tau^*b}=
  \tau^*\pb{x_i,x_j}_{\tau^*b}\;.
\end{equation*}%
It follows that $\tau^*$ is a Poisson automorphism between $\PB_b^{(k)}$ and $\PB_{\tau^*b}^{(k)}$. According to
Lemma~\ref{lma:tau}, $\tau^* C=C$, so that
\begin{equation*}
  \pb{x_2,C}_b^{(k)}=\pb{\tau^*x_1,\tau^*C}_b^{(k)}=\tau^*\pb{x_1,C}_{\tau^*b}^{(k)}\;,
\end{equation*}%
and we can conclude from $\pb{x_1,C}_b^{(k)}=0$ for all $b$ that $\pb{x_2,C}_b^{(k)}=0$ for all $b$, and
similarly that $\pb{x_j,C}_b^{(k)}=0$ for all $j\in\cI$ and for all $b$.

In order to show that $\pb{x_1,C}_b^{(k)}=0$, we decompose $C$ as $C=\sum_{\ell=0}^k(C_{\ell,+}+C_{\ell,-})$, where
\begin{equation*}
  C_{\ell,+}:=\sum_{\um\in\cS_{\ell,+}}b'_\um x_\um\;,\quad\hbox{and}\quad
  C_{\ell,-}:=\sum_{\um\in\cS_{\ell,-}}b'_\um x_\um\;.
\end{equation*}%
With this notation and in view of the definition \eqref{eq:PB_sum} of $\PB^{(k)}_b$, proving that
$\pb{x_1,C}_b^{(k)}=0$ amounts to proving that
\begin{equation*}%
  \sum_{\ell=0}^k\left(\pb{x_1,C_{\ell,+}}+\pb{x_1,C_{\ell,+}}_b+\pb{x_1,C_{\ell,-}}+
  \pb{x_1,C_{\ell,-}}_b\right)=0\;.
\end{equation*}%
We will do this by showing that
\begin{equation}\label{eq:3_eqs}
  \pb{x_1,C_{\ell,+}}=\pb{x_1,C_{\ell,+}}_b+\pb{x_1,C_{i-1,-}}=\pb{x_1,C_{\ell,-}}_b=0\;,
\end{equation}%
for all $\ell\geqs0$ (with the convention that $C_{-1,-}=0$).  We start with the first term. Let $\um\in\cS_\ell$
with $\um=(m_1,m_2,\ldots,m_{2\ell+1})$.  Recall that the sets $\cS_\ell$ are defined such that $\um\in\cS_\ell$ if
and only if $A'_{\um}=A^{(\ell)}$ (see \ref{eq:B_def}). If $\um\in\cS_{\ell,+}$ then $m_1=1$ and
$$
  \pb{x_1,x_{\um}}^{(k)}=x_1x_{\um}\left(\sum_{j=1}^{2\ell+1}A^{(k)}_{m_1,m_j}\right)=
  x_1x_{\um}\left(\sum_{j=2}^{2\ell+1}A^{(\ell)}_{1,j}\right)=0\;.
$$
This shows that $\pb{x_1,C_{\ell,+}}=0$.  We next prove the last equality in \eqref{eq:3_eqs}.  From the definition
of $\PB_b$ it follows that for any function $F$
\begin{equation}\label{eq:b_der}
  \{x_1,F\}_b=b_{1,k+1}\frac{\partial F}{\partial x_{k+1}}+b_{1,k+2}\frac{\partial F}{\partial x_{k+2}}\;.
\end{equation}
Therefore
$$
  \pb{x_1,C_{\ell,-}}_b=\sum_{\um}\pb{x_1,b'_{\um}x_{\um}}_b\;,
$$
where the sum can be restricted to those $\um\in\cS_{\ell,-}$ which contain $k+1$ or $k+2$. Since $1\notin\um$,
\eqref{eq:s-r} implies that $k+1$ and $k+2$ cannot both belong to $\um$. Also, the inequalities in Proposition
\ref{lma:S} imply that the vector $\psi(\um)$ whose effect is to replace $k+1$ in $\um$ by $k+2$ leads to a new
element of $\cS_{\ell,-}$ containing $k+2$ (and vice-versa). Therefore
\begin{eqnarray*}
  \pb{x_1,C_{\ell,-}}_b&=&
  \sum_{\stackrel{\um\in\cS_{\ell,-}}{k+1\in\um\not\ni k+2}}\pb{x_1,b'_{\um}x_{\um}}_b+
    \sum_{\stackrel{\um\in\cS_{\ell,-}}{k+1\not\in\um\ni k+2}}\pb{x_1,b'_{\um}x_{\um}}_b\\
  &=&\sum_{\stackrel{\um\in\cS_{\ell,-}}{k+1\in\um\not\ni k+2}}
    \(\pb{x_1,b'_{\um}x_{\um}}_b+\pb{x_1,b'_{\psi(\um)}x_{\psi(\um)}}_b\)\\
  &=&\sum_{\stackrel{\um\in\cS_{\ell,-}}{k+1\in\um\not\ni k+2}}\left(\frac{b_{1,k+1}b'_{\um}x_{\um}}{x_{k+1}}+
    \frac{b_{1,k+2}b'_{\psi(\um)}x_{\psi(\um)}}{x_{k+2}}\right)=0\;,
\end{eqnarray*}
since $x_\um/x_{k+1}= x_{\psi(\um)}/x_{k+2}$ and
$$
b_{1,k+1}b'_{\um}=b'_{1,k+1}b'_{\um}=b'_{1,k+2}b'_{\psi(\um)}=-b_{1,k+2}b'_{\psi(\um)}
$$
for all $\um\in\cS_{\ell,-}$ with $k+1\in\um$ and $k+2\not\in\um$.

It remains to be shown that $\pb{x_1,C_{\ell,+}}_b+\pb{x_1,C_{i-1,-}}=0$.
Using \eqref{eq:b_der}, we find
\begin{equation}\label{eq:cas_1}
  \pb{x_1,C_{\ell,+}}_b=
  \sum_{\stackrel{\un\in\cS_{\ell,+}}{k+1\in\un}} b_{1,k+1}b'_\un\frac{x_{\un}}{x_{k+1}}+
  \sum_{\stackrel{\un\in\cS_{\ell,+}}{k+2\in\un}} b_{1,k+2}b'_\un\frac{x_{\un}}{x_{k+2}}\;.
\end{equation}
For $\um\in\cS_{\ell-1,-}$, if $m_\ell\leq k+1$ then $\pb{x_1,x_\um}^{(k)}=x_1x_\um\left(\sum_{j=1}^{2\ell-1}
A_{1,m_j}\right)=x_1x_\um$, and similarly if $m_\ell\geq k+2$ then $\pb{x_1,x_\um}^{(k)}=-x_1x_\um$. Therefore,
\begin{equation}\label{eq:cas_2}
  \pb{x_1,C_{\ell-1,-}}^{(k)}=-\sum_{\stackrel{\um\in\cS_{\ell-1,-}}{m_\ell>k+1}}b'_\um x_1x_\um+
  \sum_{\stackrel{\um\in\cS_{\ell-1,-}}{m_\ell\leqs k+1}}b'_\um x_1x_\um\;.
\end{equation}
Up to a minus sign, the first sum in \eqref{eq:cas_1} and the first sum in \eqref{eq:cas_2} are equal, and
similarly for the second sums. We show this for the first sum, using the bijection $\phi_2$ from Lemma
\ref{lma:phi_maps} (for the second sum, one uses $\phi_1$):
\begin{equation*}
  \begin{array}{lcccc}
    \phi_2&:&\set{\um\in\cS_{\ell-1,-}\mid m_\ell\geqs k+2}&\to&\set{\un\in\cS_{\ell,+}\mid k+1\in\un}\\
    & & \um&\mapsto&\um\oplus(1,k+1)\;.
  \end{array}
\end{equation*}
For $\un:=\phi_2(\um)=\um\oplus(1,k+1)$, with $\um\in\set{\um\in\cS_{\ell-1,-}\mid m_\ell\geqs k+2}$, we have
$x_\un=x_\um x_1x_{k+1}$ and $b'_\un=b'_\um/b'_{1,k+1}=b'_\um/b_{1,k+1}$, so that
\begin{equation*}
  \sum_{\stackrel{\un\in\cS_{\ell,+}}{k+1\in\un}} b_{1,k+1}b'_\un\frac{x_{\un}}{x_{k+1}}=
  \sum_{\stackrel{\um\in\cS_{\ell-1,-}}{m_\ell>k+1}}b'_\um x_1x_\um\;.
\end{equation*}%
It follows that the right hand sides of \eqref{eq:cas_1} and \eqref{eq:cas_2} sum up to zero,
as was to be shown.
\end{proof}


\section{Lax equation, first integrals and integrability}\label{sec:Lax}
We now introduce a family of deformations of the Bogoyavlenskij-Itoh Lotka-Volterra systems, which we will show to
be Liouville integrable. Let $c_1,c_2,\dots,c_{2k+1}$ be arbitrary constants, satisfying $\sum_{i=1}^{2k+1}
c_i=0$. The \emph{deformed Bogoyavlenskij-Itoh Lotka-Volterra system} is the system defined by the following
differential equations:
\begin{equation}\label{eq:LV_defo}
  \dot x_i = \sum_{j=1}^{2k+1} A_{i,j} x_i x_j+c_i\;, \ \ i=1,2, \dots , 2k+1 \; ,
\end{equation}
where the constants $A_{i,j}$ are given by (\ref{eq:bogo_mat}).  We first show that (\ref{eq:LV_defo}) is
Hamiltonian with respect to some of the Poisson structures $\PB^{(k)}_b$, introduced in Section
\ref{sec:poisson_defo}, with $H=\sum_{i=1}^{2k+1}x_i$ as Hamiltonian function. Before doing this, notice that if
(\ref{eq:LV_defo}) admits $H$ as a Hamiltonian function, then $H$ is a constant of motion of (\ref{eq:LV_defo}),
and so $\sum_{i=1}^{2k+1} c_i=0$, which explains why we have imposed the latter condition on the deformation
constants $c_i$. Let $b$ be as in Proposition \ref{prop:compatible}, i.e., $b$ is skew-symmetric and $b_{i,j}=0$
for all $i$ and $j$ such that $\vert{i-j}\vert\notin\set{k,k+1}$. Then the Hamiltonian vector field
$\Pb{H}^{(k)}_b$ is given by
\begin{equation}\label{eq:LV_ham}
  \dot x_i = \sum_{j=1}^{2k+1} A_{i,j} x_i x_j+b_{i,i+k}-b_{i-k,i}\;, \ \ i=1,2, \dots , 2k+1 \;,
\end{equation}
where, as above, all indices are taken modulo $2k+1$ and with values in $\cI=\set{1,2,\dots,2k+1}$.  Comparing
(\ref{eq:LV_defo}) and (\ref{eq:LV_ham}), we need to show that the following linear system admits a solution:
\begin{equation}\label{eq:lin_to_solve}
  b_{i,i+k}-b_{i-k,i}=c_i\;, \ \ i\in\cI\;.
\end{equation}%
Let $b_{1-k,1}:=c$ be arbitray. Then the solution to the equations (\ref{eq:lin_to_solve}) can be obtained by
solving them for the following values of $i$ (in that order and, as always, modulo $2k+1$): $i=1,k+1,2k+1,\dots$,
yielding unique values for $b_{1,1+k}, b_{1+k,1+2k},$ $b_{1+2k,1+3k}=b_{1+2k,k},\dots$. The final formula is given by
\begin{equation}\label{eq:b_sol}
  b_{i,i+k}=c_1+c_{1+k}+c_{1+2k}+\cdots+c_{i-k}+c_i+c\;, \quad\hbox{for}\quad i\in\cI\;.
\end{equation}
indeed, it follows from the latter formula at once that $b_{i,i+k}-b_{i,i-k}=c_i$ for all $i\in\cI$ and that
$$
  b_{1-k,1}=c_1+c_{1+k}+c_{1+2k}+\cdots+c_{2}+c_{2+k}+c=\sum_{i=1}^{2k+1}c_i+c=c\;,
$$
since $\sum_{i=1}^{2k+1} c_i=0$. We have thereby proven the following proposition:
\begin{prp}
  The deformed Bogoyavlenskij-Itoh Lotka-Volterra system
  \begin{equation*}
    \dot x_i = \sum_{j=1}^{2k+1} A_{i,j} x_i x_j+c_i\;, \ \ i=1,2, \dots , 2k+1 \; ,
  \end{equation*}
  with $\sum_{i=1}^{2k+1}c_i=0$ is a Hamiltonian system with respect to a (non-unique) Poisson structure
  $\PB^{(k)}_b$, the non-zero entries of $b$ being given by (\ref{eq:b_sol}) and with $H=\sum_{i=1}^{2k+1}x_i$ as
  Hamiltonian function.
\end{prp}
In the following proposition we give a Lax equation (with spectral parameter) for the deformed Bogoyavlenskij-Itoh
Lotka-Volterra system (\ref{eq:LV_defo}).  Bogoyavlenskij's Lax equation (\ref{eq:bogo_lax}) is obtained from it by
putting all $b_{i,j}$ equal to zero.
\begin{prp}
  The system (\ref{eq:LV_defo}) can be written as the following Lax equation (with spectral parameter~$\l$) 
  \begin{equation}\label{eq:bogo_deformed_lax}
    (X+\lambda^{-1}\Delta+\lambda M)^\cdot=[X+\lambda^{-1}\Delta+\lambda M,B-\lambda M^{k+1}]
  \end{equation}
  where for $i,j\in \cI$ the $(i,j)$-th entry of the matrices $X,\,M$ and $B$ is respectively given by
  \begin{gather}
    X_{i,j}:=\delta_{i,j+k}x_i\;,\quad \Delta_{i,j}:=b_{i+k,j}\delta_{i,j}\;,\quad M_{i,j}:=\delta_{i+1,j}\;,
    \label{eq:bogo_defo_lax_entries_1}\\
    B_{i,j}:=b_i\delta_{i,j}:=-\delta_{i,j}(x_i+x_{i+1}+\cdots+x_{i+k})\;.\label{eq:bogo_defo_lax_entries_2}
  \end{gather}%
\end{prp}
\begin{proof}
To check that (\ref{eq:bogo_lax}) is equivalent to (\ref{eq:LV_ham}) it is sufficient to check that
(\ref{eq:LV_ham}) is equivalent with $\dot X=[X,B]-\lb{\Delta,M^{k+1}}$ and (since $M$ is constant) that
$[M,B]-[X,M^{k+1}]=0$; indeed, $[M,M^{k+1}]=0$ and $[\Delta,B]=0$ because $\Delta$ and $B$ are diagonal matrices.
For the second equality, one finds at once from (\ref{eq:bogo_defo_lax_entries_1}) that
\begin{equation*}
  ([M,B]-[X,M^{k+1}])_{i,j}=\delta_{i+1,j}(b_j-b_i-x_i+x_{j+k})=0\;.
\end{equation*}%
Since $B$ and $\Delta$ are diagonal matrices, $[X,B]_{i,j}=X_{i,j}(b_j-b_i)$ and
$-\lb{\Delta,M^{k+1}}_{i,j}=(b_{i,i+k}-b_{i-k,i})\delta_{i,j+k}$, with non-zero entries only when $j=i-k$; for
these entries, one has from the Lax equation
\begin{equation*}
  \dot x_i=\dot X_{i,i-k}=[X,B]_{i,i-k}-\lb{\Delta,M^{k+1}}_{i,i-k}=x_i(b_{i-k}-b_i)+b_{i,i+k}-b_{i-k,i}\;,
\end{equation*}%
which is the right hand side of (\ref{eq:LV_ham}). 
\end{proof}

Each coefficient of the characteristic polynomial of the Lax operator
\begin{equation}\label{eq:lax_operator}
  L^b(\l):=X+\lambda^{-1}\Delta+\lambda M
\end{equation}
gives a first integral for the deformed Bogoyavlenskij-Itoh Lotka-Volterra system. We know already one constant of
motion: the Casimir $K_k^b$, introduced in Proposition \ref{prop:deformed_casimir}. Indeed, $\dot
K_k^b=\pb{K_k^b,H}_b^{(k)}=0$. We will show in the Proposition \ref{prp:determinant_Lax} that $K_k^b$ appears
as a coefficient of the characteristic polynomial of the Lax operator $L^b(\l)$ (namely the coefficient of $\l^0$).
But first we prove a basic lemma on permutations, which will be used in the Proposition \ref{prp:determinant_Lax}
and will clarify its proof.
\begin{lem}\label{lma:permutations}
  For a permutation $\s$ of $\cI$, the following two conditions are equivalent:
  \begin{enumerate}
    \item[(i)] $\s(i)\in\set{i,i+k,i+k+1}$ for all $i\in\cI$;
    \item[(ii)] $\s$ is equal to one of the following:
      \begin{enumerate}
        \item[(1)] The $(2k+1)$-cycle $(1,1+k,1+2k,\dots,2,2+k)$;
        \item[(2)] The inverse of the $(2k+1)$-cycle $(1,1+k,1+2k,\dots,2,2+k)$;
        \item[(3)] A product of $k-\ell$ transpositions $(r_j,s_j)$ with disjoint support, and satisfying
          $s_j-r_j\in\set{k,k+1}$ for $j=1,\dots,k-\ell$.  
      \end{enumerate}
  \end{enumerate}
\end{lem}
\begin{proof}
The implication $(ii)\implies (i)$ is clear. We prove the other implication. Suppose that
$\s(i)\in\set{i,i+k,i+k+1}$ for all $i\in\cI$ and that $\s$ is not as stated in (3) above. Then there exist
$j\in\cI$ such that $\s^2(j)\neq j$. Since $\s(j)\in \set{i,i+k,i+k+1}$, we have either $\s(j)=j+k$ or
$\s(j)=j+k+1$. By replacing $\s$ with $\s^{-1}$, if needed, we may suppose that $\s(j)=j+k$. Since $\s^2(j)\neq
j\ (=j+2k+1)$, we must have $\s^2(j)=\s(j+k)=j+2k$. By recursion, $\s^{s}(j)=j+s k$ for $s=0,1,2,\dots$. Since $k$
and $2k+1$ are relatively prime,
\begin{equation*}
  \s=(j,\; j+k,\; j+2k,\; \dots,\; j+1,\; j+1+k)=(1,\; 1+k,\; 1+2k,\;\dots,\; 2,\; 2+k)\;.
\end{equation*}%
This proves that $\s$ is of the form $(1)$.
\end{proof}

\begin{prp}\label{prp:determinant_Lax}
  The determinant of the Lax operator $L^b(\l)$ is given by
  \begin{equation}\label{eq:det_Lax}
    \det L^b(\l)=\l^{2k+1}+\frac{1}{\lambda^{2k+1}}\prod_{j=1}^{2k+1}b_{j+k,j}+K_k^b\;,
  \end{equation}
  where $K_k^b$ is the Casimir (\ref{eq:defo_cas}) of $\PB_b^{(k)}$.
\end{prp}
\begin{proof}
Let $\Lambda:=L^b(\l) E$, where $E$ denotes the $(2k+1)\times (2k+1)$-matrix with entries $E_{j,j+k}=1$ for all
$j\in\cI$ and all other entries equal to zero. Thus, $\Lambda$ is just the matrix $L^b(\l)$ with its last $k$
columns placed first. The nonzero entries of $\Lambda$ are the following $(j\in\cI)$:
\begin{enumerate}
  \item[(1)] $\Lambda_{i,i}=x_i$;
  \item[(2)] $\Lambda_{i,i+k}=\frac{b_{i+k,i}}{\lambda}$;
  \item[(3)] $\Lambda_{i,i+k+1}=\lambda$.
\end{enumerate}
Also, $\det\Lambda=\det L^b(\l)$ since $\det E=1$. We compute this determinant by using the Leibniz formula:
$$
  \det(\Lambda)=\sum_{\s\in S_{\cI}}\operatorname{sgn}(\s)\prod_{i=1}^{2k+1}\Lambda_{i,\s(i)}\;,
$$
where $S_\cI$ is the permutation group of $\cI$ and $\operatorname{sgn}(\s)$ stands for the sign of a permutation
$\s$. Because of items (1) -- (3) above, $\prod_i\Lambda_{i,\s(i)}\neq 0$ if and only if
$\s(i)\in\set{i,i+k,i+k+1}$ for every $i\in\cI$. According to Lemma \ref{lma:permutations}, there are three
possibilities for a permutation $\s$ satisfying these conditions. For the first one,
$\s=(1,1+k,1+2k,\dots,2,2+k)$. Since $\operatorname{sgn}(\s)=1$, it leads to the following contribution in
$\det\Lambda$:
\begin{equation*}
  \prod_{i=1}^{2k+1}\Lambda_{i,i+k}=\frac1{\l^{2k+1}}\prod_{j\in\cI}b_{j+k,j}\;.
\end{equation*}%
For the second one, $\s^{-1}=(1,1+k,1+2k,\dots,2,2+k)$, leading to $\l^{2k+1}$ since again
$\operatorname{sgn}(\s)=1$ and since $\Lambda_{i,i+k+1}=\lambda$ for all $i$. For the third and final one, $\s$ is
the product of transpositions,
$$
\s=(r_1,\ s_1)(r_2,\ s_2)\cdots(r_{k-\ell},\ s_{k-\ell})\;,
$$
with $s_j-r_j\in\set{k,k+1}$ for $j=1,\dots,k-\ell$. In view of
Proposition \ref{prp:m'} we define $\um\in\cS_\ell$ to be the element
$\um'=(r_1,r_2,\dots,r_{k-\ell},s_1,s_2,\dots,s_{k-\ell})\in\cS_\ell$. The
corresponding term in $\det(\Lambda)$ is
%
%
\begin{equation*}
  \operatorname{sgn}(\s)\prod_{j=1}^{k-\ell}(\Lambda_{r_j,s_j}\Lambda_{s_j,r_j})\prod_{j=1}^{2\ell+1} x_{m_j}
  =\prod_{j=1}^{k-\ell}b'_{r_j,s_j}x_{\um}=b'_\um x_\um\;.
\end{equation*}%
Summing up all these contributions, we get in view of Definition \ref{eq:defo_cas} of $K_k^b$,
precisely (\ref{eq:det_Lax}).
\end{proof}

We can obtain from the above formula (\ref{eq:det_Lax}) for the determinant of the Lax operator $L^b(\l)$ easily a
formula for its characteristic polynomial $\det(L^b(\l)-\mu\Id)$ by the following trick: since the $j$-th diagonal
entry of $L^b(\l)$ is $b_{j+k,j}/\l$ and since the parameters $b_{i,j}$ appear nowhere else in the Lax matrix, it
suffices to replace in all formulas $b_{j+k,j}/\l$ by $b_{j+k,j}/\l-\mu$. This can be done by replacing every
$b_{j+k,j}$ by $b_{j+k,j}-\l\mu$. Notice that when the constants $b_{j+k,j}$ are viewed as a vector $b$ (with
$2k+1$ entries), then this substitution amounts to replacing $b$ with $b-\l\mu\bfone$, where
$\bfone=(1,1,\dots,1)$. This proves part of the following proposition:
\begin{prp}\label{prp:char_Lax}
  The characteristic polynomial of the Lax operator $L^b(\l)$ is given by
  \begin{equation}\label{eq:char_poly_Lax}
    \chi(L^b(\l),\mu)=\det(L^b(\l)-\mu\Id)=\l^{2k+1}+\frac{1}{\lambda^{2k+1}}\prod_{j=1}^{2k+1}(b_{j+k,j}-\l\mu)+
    K_k^{b-\l\mu\bfone}\;.
  \end{equation}%
  The coefficients $K_\ell^b$ in the expansion
  \begin{equation}\label{eq:casimir_expansion}
    K_k^{b-\l\mu\bfone}=\sum_{\ell=0}^k (\l\mu)^{k-\ell}K_\ell^b\;,
  \end{equation}%
  are first integrals of the deformed Bogoyavlenskij-Itoh Lotka-Volterra system (\ref{eq:LV_ham}). Also, each
  $K_\ell^b$ is of total degree $2\ell+1$ (as a polynomial in $x_1,\dots,x_{2k+1}$), with leading term $K_\ell$.
\end{prp}
\begin{proof}
Formula (\ref{eq:char_poly_Lax}) was proven above. In order to prove the other statements, we first point out that
if we attribute the weight $2$ to every $b_{j+k,j}$ and the weight $1$ to all variables $x_i$, to $\l$ and to $\mu$
then every entry of the matrix $L^b(\l)-\mu\Id$ is homogeneous of degree 1, and so the characteristic polynomial
(\ref{eq:char_poly_Lax}) is weight homogeneous of degree $2k+1$ with respect to these weights. It follows that
$K_k^{b-\l\mu\bfone}$ can indeed be expanded as in (\ref{eq:casimir_expansion}) and that every coefficient
$K_\ell^b$ is weight homogeneous of degree $2\ell+1$ (depending polynomially on the variables $x_i$ and the
constants $b_{j+k,j}$). Since these polynomials are (possibly up to a constant) coefficients of the characteristic
polynomial of the Lax operator, they are constants of motion. Setting $b=0$ in (\ref{eq:char_poly_Lax}) and in
(\ref{eq:casimir_expansion}) we have
\begin{equation*}
  \chi(L(\l),\mu)=\l^{2k+1}-\mu^{2k+1}+K_k^{-\l\mu\bfone}=
  \l^{2k+1}-\mu^{2k+1}+\sum_{\ell=0}^k (\l\mu)^{k-\ell}K_\ell^0\;,
\end{equation*}%
which, compared with (\ref{eq:bogo_char_poly}), to wit,
\begin{equation*}
  \chi(L(\l),\mu)=\l^{2k+1}-\mu^{2k+1}+\sum_{\ell=0}^k(\l\mu)^{k-\ell}K_\ell\;,
\end{equation*}%
shows that $K_\ell$ is obtained from $K_\ell^b$ by putting all constants in $b$ equal to zero; since the constants
in $b$ are of weight $2$ and both $K_\ell$ and $K_\ell^b$ are weight homogeneous of degree $2\ell+1$, it follows
that $K_\ell^b$ is of total degree $2\ell+1$ (as a polynomial in $x_1,\dots,x_{2k+1}$), with leading term $K_\ell$,
as asserted.
\end{proof}

We have constructed $k+1$ polynomial first integrals $K_0^b,K_1^b,\dots,K_k^b$ for the deformed Bo\-go\-yav\-len\-skij-Itoh
Lotka-Volterra system~(\ref{eq:LV_defo}), which is a Hamiltonian system on the Poisson manifold $(\bbR^{2k+1},\PB_b^{(k)})$.
We now prove the Liouville integrability of this system.
\begin{teo}
  The deformed Bogoyavlenskij-Itoh Lotka-Volterra system (\ref{eq:LV_defo}) is Liouville integrable, with
  independent first integrals $K_0^b,K_1^b,\dots,K_k^b$ which are pairwise in involution with respect to
  $\PB_b^{(k)}$.
\end{teo}
\begin{proof}
According to Proposition  \ref{prp:char_Lax}, the $k+1$ polynomials $K_\ell^b$, defined by 
\begin{equation}\label{eq:exp_nu}
  K_k^{b-\nu\bfone}=\sum_{\ell=0}^k \nu^{k-\ell}K_\ell^b\;
\end{equation}%
are first integrals of (\ref{eq:LV_defo}). By the same proposition, the leading term of $K_\ell^b$ (as a polynomial
in $x_1,\dots,x_{2k+1}$) is $K_\ell$; since the polynomials $K_0,K_1,\dots,K_k$ are (functionally) independent, the
latter fact implies that the same is true for the deformed polynomials $K_0^b,K_1^b,\dots,K_k^b$. It remains to be
shown that the latter polynomials are in involution with respect to $\PB_b^{(k)}$. To do this, we use the
Lenard-Magri scheme in the context of Poisson pencils (see \cite{magri}). Recall from Proposition
\ref{prop:deformed_casimir} that $K_k^b$ is a Casimir for $\PB_b^{(k)}$. It follows that for any $\nu$,
$K_k^{b-\nu\bfone}$ is a Casimir of $\PB_{b-\nu\bfone}^{(k)}=\PB_{b}^{(k)}-\nu\PB_\bfone$. Said differently,
$K_k^{b-\nu\bfone}$ is a Casimir for the Poisson pencil $\PB_{b}^{(k)}-\nu\PB_\bfone$. Its expansion
(\ref{eq:exp_nu}) implies according to the Lenard-Magri scheme the following facts:
\begin{enumerate}
  \item[(1)] The polynomials $K_\ell^b$ are in involution with respect to $\PB_b^{(k)}$;
  \item[(2)] They are also in  involution with respect to $\PB_\bfone$;
  \item[(3)] Each of the integrable vector fields is bi-Hamiltonian:
    $$
    \Pb{K_\ell^b}_b^{(k)}=\Pb{K_{\ell+1}^b}_\bfone\,;
    $$ 
  \item[(4)]$K_k^b$ is a Casimir of $\PB_b^{(k)}$ (as we already know) and $K_0^b$, which is just $H$, is a Casimir
    of $\PB_\bfone$ (which is quite obvious).
\end{enumerate}
In particular, the first integrals are in involution with respect to $\PB_b^{(k)}$, as was to be shown.
\end{proof}

\begin{rem}
The proof shows that the functions $K_\ell^b, \ell=0,1,\ldots, k$ are also in involution with respect to the bracket
$\PB_{\bfone}$.
\end{rem}

\begin{rem}
For $b=0$, we get that the functions $K_\ell, \ell=0,1,\ldots, k$ are in involution with respect to the brackets
$\PB^{(k)}$, which provides an alternative proof to Itoh's combinatorial proof of this fact.
\end{rem}


\section{Relation with the Veselov-Shabat system}\label{sec:veselov}

In this section we will give a recursive formula for constructing the first integrals $K_\ell^b$ which shows that
the combinatorial argument of Itoh can be used for their definition as in the case of the polynomials $K_\ell$,
the first integrals of the Bogoyavlenskij-Itoh system (see \cite{PPPP}). We will also show that the system
\eqref{eq:LV_ham} and the Veselov-Shabat system constructed in \cite{VS} are isomorphic by constructing a Poisson
isomorphism between $(\mathbb R^{2k+1},\PB_{b}^{(k)})$ and $(\mathbb R^{2k+1},\PB_v)$ where $\PB_v$ is the Poisson
structure constructed in \cite{VS} as part of the Hamiltonian formalism of the Veselov-Shabat system.
The Poisson structure $\PB_v$ is defined by the formulas $\pb{g_i,g_j}_v=(-1)^{j-i+1}g_ig_j$ if $j>i+1$ and
$\pb{g_i,g_{i+1}}_v=g_ig_{i+1}+\beta_{i+1}$ and the Veselov-Shabat system is the Hamiltonian system with Poisson
structure $\PB_v$ and Hamiltonian the sum of the variables $\sum_{j=1}^ng_j$.

We define a permutation $\rho$ of $\cI$ by setting $\rho(\ell):=\rho_{\ell}=(\ell-1)k+1$ for all $\ell\in\mathbb Z$.
Then, one defines the linear map $g:(\mathbb R^{2k+1},\PB_b^{(k)})\rightarrow (\mathbb R^{2k+1},\PB_v)$ by
$$
(x_1,x_2,\ldots,x_{2k+1})\mapsto (g_1,g_2,\ldots,g_{2k+1}):=(x_{\rho_1},x_{\rho_2},\ldots,x_{\rho_{2k+1}})
$$
which is easily checked to be a Poisson isomorphism. The parameters $b_{i,i+k}$ of the first Poisson structure are
related to the parameters $\beta_i$ of the second Poisson structure by $\beta_{i+1}=b_{\rho_i,\rho_i+k}$ for all
$i=1,2,\ldots,2k+1$. Since the Hamiltonian $\sum_{j=1}^nx_j$ of the system \eqref{eq:LV_ham} is mapped under $g$ to
the Hamiltonian $\sum_{j=1}^ng_j$ of the Veselov-Shabat system, the two systems are isomorphic.

Following \cite{VS} we define new variables $f_i$ which are related to the $g_i$ by the formulas $g_i=f_i+f_{i+1}$
or equivalently by $f_i=\frac{1}{2}\sum_{j=1}^{2k+1}(-1)^{j+1}g_{i+j-1},\; i=1,2,\ldots, 2k+1$.  Then the system of
Veselov and Shabat in the $f_i$ variables reads
\begin{equation}
\label{eq:VS_sys}
\dot{f}_i+\dot{f}_{i+1}=f_{i+1}^2-f_i^2+\beta_{i+1}-\beta_i,\quad i=1,2,\ldots,2k+1
\end{equation}
and the matrix equation $\dot{L}_i=L_i U_{i+1}-U_i L_i$ where
\begin{equation*}
L_i=
\begin{pmatrix}
f_i & 1 \\
f_i^2+\beta_i-\lambda & f_i
\end{pmatrix}, \qquad
U_i=
\begin{pmatrix}
0 & 1 \\
F_i+f_i^2+\beta_i-\lambda & 0
\end{pmatrix}
\end{equation*}
with
\begin{gather*}
F_i=g_{i+1}(g_{i}-g_{i+2}-g_{i+4}-\cdots-g_{i+2k})+
g_{i+3}(g_{i}+g_{i+2}-g_{i+4}-\cdots-g_{i+2k})+\\
g_{i+2k-1}(g_{i}+g_{i+2}+g_{i+4}+\cdots+g_{i+2k-2}-g_{i+2k}),
\end{gather*}
gives the $i$-th equation of the system \eqref{eq:VS_sys}.
Notice that $U_{2k+2}=U_1$ and therefore if we write $L=\prod_{i=1}^nL_i$
we have that
\begin{align*}
\dot{L}&=\sum_{i=1}^{2k+1}L_1L_2\ldots\dot{L}_i\ldots L_{2k+1}\\
&=\sum_{i=1}^{2k+1}L_1L_2\ldots L_iU_{i+1}\ldots L_{2k+1}-
\sum_{i=1}^{2k+1}L_1L_2\ldots U_iL_i\ldots L_{2k+1}\\
&=LU_{2k+2}-U_1L=[L,U_1]\, ,
\end{align*}
which is a Lax equation for the system \eqref{eq:VS_sys}. The simple form of the matrices $L_i$ allows us to compute
easily the trace of $L$ and as it turns out the trace is equal to the Casimir $K_k^{b-\nu}$.
Indeed, if $R_i$ is the matrix $R_i=(\beta_i-\lambda)\begin{pmatrix}0 & 0 \\ 1 & 0 \end{pmatrix}$
then one can show that the matrices $L_i, R_i$ and $U_i$ satisfy the following relations:
\begin{gather*}
\prod_{i=j}^\ell(L_i-R_i)=\prod_{i=j}^{\ell-1} g_i\begin{pmatrix}f_\ell & 1 \\ f_jf_\ell & f_j \end{pmatrix},
\; R_iR_j=0, \\
\begin{pmatrix}f_\ell & 1 \\ f_jf_\ell & f_j \end{pmatrix}
\begin{pmatrix}0 & 0 \\ 1 & 0 \end{pmatrix}
\begin{pmatrix}f_{2k+1} & 1 \\ f_{\ell+2}f_{2k+1} & f_{\ell+2} \end{pmatrix}=
\begin{pmatrix}f_{2k+1} & 1 \\ f_jf_{2k+1} & f_j \end{pmatrix}.
\end{gather*}
From the previous relations we get a formula for the matrix $L$ in terms of the matrices
$L_i-R_i$, $i=1,2,\ldots,2k+1$. Since the product of any two matrices $R_i, R_j$
is zero it follows that the matrix $L=\prod_{i=1}^{2k+1}(L_i-R_i+R_i)$ is sum
of products of $L_i-R_i$ for $i=1,2,\ldots,2k+1$ with no consecutive gaps between them,
i.e. a sum of elements of the form
$$
(L_1-R_1)\ldots(L_{i-1}-R_{i-1})R_i(L_{i+1}-R_{i+1})\ldots(L_{j-1}-R_{j-1})R_j(L_{j+1}-R_{j+1})\ldots
$$
with $i<j-1<\ldots$. The above formulas allow us to write such products as a single matrix
and since the trace is linear we get that
\begin{equation}\label{eq:VS_form}
\operatorname{tr}(L)=
\prod_{i=1}^{2k+1}(1+(\beta_{i+1}-\lambda)\frac{\partial^2}{\partial g_i\partial g_{i+1}})\prod_{j=1}^{2k+1}g_i.
\end{equation}
The trace of $L$, written in the variables $x_i$, becomes
\begin{align*}
\operatorname{tr}(L)&=\prod_{i=1}^{2k+1}(1+(b_{\rho_{i},\rho_{i}+k}-\lambda)\frac{\partial^2}{\partial x_{\rho_i}\partial x_{\rho_{i+1}}})\prod_{j=1}^{2k+1}x_{\rho_i}\\
&=\prod_{i=1}^{2k+1}(1+(b_{i,i+k}-\lambda)\frac{\partial^2}{\partial x_i\partial x_{i+k}})\prod_{j=1}^{2k+1}x_i.
\end{align*}
Since L is a Lax operator, the coefficients of the trace of $L$ (viewed as a polynomial in $\lambda$) are first integrals of \eqref{eq:LV_ham}.
We will show that the previous formula for the trace of $L$ is exactly the polynomial $K_k^{b-\lambda}$
and we will also give a recursive formula for explicitly producing the polynomials $K_j^b$. To do this, we first define the
following constant coefficient linear operators
$$
  D_{i,j}=\frac{\partial^2}{\partial x_i\partial x_j}, \quad D=\sum_{i=1}^{2k+1}D_{i,i+k}=\sum_{i=1}^{2k+1}\frac{\partial^2}{\partial x_i\partial x_{i+k}}
$$
and
$$
D_b=\sum_{i=1}^{2k+1}b_{i,i+k}\frac{\partial^2}{\partial x_i\partial x_{i+k}}\,.
$$
Since they have constant coefficients, all these operators commute with each other (at least when applied on
polynomials, which is the case which we consider here). In the following lemma we show that the first
integrals $K_i$ of \eqref{eq:LV_ham} can be defined recursively using the operator $D$.

\begin{lem}
\label{lem:rec_int}
The polynomials $K_i$, defined by \eqref{eq:bogo_char_poly} (which are the same thing as $K_i=K_i^{\bf 0}$)
satisfy the following recursion relation
$$
  D K_i=(k-i+1) K_{i-1}.
$$
\end{lem}
\begin{proof}
Each element $\um'\in\cS_i'$ (and therefore each element $\um\in\cS_i$) is in one to one correspondence
with the $i'$-tuples consisting of elements $b_{\ell,\ell+k}, \ell=1,2,\ldots,n$ which do not have
any common indices (e.g. the choice $b_{1,5}, b_{5,9}, \ldots$ is not allowed).
This is obvious from Proposition \ref{prp:m'}. Therefore there is a bijection $\Phi:\cS_i\rightarrow \cB_{i'}$
where
$$
\cB_{i'}=\set{\set{b_{r_1,r_1+k}, \ldots, b_{r_{i'},r_{i'}+k}}\mid
\#\set{r_1,\ldots,r_{i'},r_{1}+k,\ldots,r_{i'}+k}=2i'}.
$$
We can easily see that the map $\Phi$ is the complement, meaning that
\begin{gather*}
\Phi(\um)=\set{b_{r_1,r_1+k}, b_{r_2,r_2+k}, \ldots, b_{r_{i'},r_{i'}+k}}\iff\\
\um'=(r_1,r_2,\ldots,r_{i'},r_{1}+k,r_2+k,\ldots,r_{i'}+k).
\end{gather*}
Let $\um=(m_1,m_2,\ldots,m_{2i+1})$ be an element of the set $\cS_i$ such that
$\Phi(\um)=\set{b_{r_1,r_1+k}, b_{r_2,r_2+k}, \ldots, b_{r_{i'},r_{i'}+k}}$.
Then
$
x_{\um}=
\left(\prod_{\ell=1}^{i'}D_{r_\ell,r_\ell+k}\right)K_k
$
and, using the combinatorial formulas \eqref{eq:defo_cas} we get
$$
K_i=\sum_{\set{b_{r_1,r_1+k}, b_{r_2,r_2+k}, \ldots, b_{r_{i'},r_{i'}+k}}\in\cB_{i'}}
\left(\prod_{\ell=1}^{i'}D_{r_\ell,r_\ell+k}\right)K_k.
$$
Because $\deg_{x_i}K_k=1$, it follows that $\frac{\partial^2}{\partial x_j^2}(K_k)=0$ for any $j\in\cI$.
The multinomial formula
$$
D^{i'}=\left(\sum_{j=1}^nD_{j,j+k}\right)^{i'}=\sum_{m_1+m_2+\ldots,m_k=i'}{i'\choose m_1,m_2,\ldots,m_\ell}
\prod_{j=1}^\ell D_{j,j+k}^{m_j}
$$
gives that
\begin{gather*}
D^{i'}K_k=\sum_{\set{b_{r_1,r_1+k}, b_{r_2,r_2+k}, \ldots, b_{r_{i'},r_{i'}+k}}\in\cB_{i'}}
{i'\choose 1,1,\ldots,1}\left(\prod_{j=1}^{i'} D_{r_j,r_j+k}\right)K_k=i'!K_i
\end{gather*}
and we deduce that
\begin{equation}
\label{eq:rec_K_k}
D^{i'}K_k=i'!K_i
\end{equation}
Therefore,
$$
K_{i-1}=\frac{1}{(i'+1)!}D^{i'+1}K_k=\frac{i'!}{(i'+1)!}DK_i=\frac{1}{i'+1}DK_i.
$$
\end{proof}

We prove now a similar formula for the polynomials $K_i^b$
which shows that the formula \eqref{eq:VS_form} of Veselov and Shabat
for the trace of the Lax operator $L$ coincides with the polynomial
$K_k^{b-\lambda}$ (item (1) of the next proposition) which is the
characteristic polynomial of the Lax operator we constructed in
\eqref{eq:bogo_deformed_lax}.
\begin{prp}\ %
\begin{enumerate}
\item[(1)] For each $i=0,1,2,\ldots,k$
$$
K_i^b=e^{D_b}(K_i)=K_i+D_bK_i+\frac{1}{2!}D_b^2K_i+\cdots+\frac{1}{i!}D_b^iK_i
$$
\item[(2)] For each $i=0,1,2,\ldots,k-1$
$$
K_i^b=\frac{1}{k-i}D K^b_{i+1}.
$$
\end{enumerate}
\end{prp}
\begin{proof}
First of all, note that (1) holds for $i=k$. This can be seen from the proof of
Lemma \ref{lem:rec_int} and the definition of the polynomial $K_k^b$; in the polynomial
$K_k^b$, for each variable missing from $x_{\um}$ there is one and only one constant
$b_{.,.}$ associated to it
(see also Proposition \ref{prp:m'}).

We also notice that since
    $D_{b+\l{\bf 1}}=D_b+\l D$, its $\ell$-th power (for $\ell=1,\dots,k$) is given by
$$
D_{b+\l{\bf 1}}^\ell=\sum_{m=0}^\ell{\ell \choose m}\l^mD_b^{\ell-m}D^m.
$$
Therefore, for item (1) we have
$$
K_k^{b+\l{\bf 1}}=\sum_{\ell=0}^k\frac{1}{\ell!}D_{b+\l{\bf 1}}^\ell K_k=
\sum_{\ell=0}^k\frac{1}{\ell!}\sum_{m=0}^\ell{\ell \choose m}\l^mD_b^{\ell-m}D^mK_k.
$$
Since $K_i^b$ is the coefficient of $\l^{k-i}$ in the expansion of
$K_k^{b+\l{\bf 1}}$ it follows that
$$
K_i^b=\sum_{\ell=0}^k\frac{1}{\ell!}{\ell \choose k-i}D_b^{\ell-k+i}D^{k-i}K_k=
\sum_{\ell=k-i}^k\frac{1}{\ell!}{\ell \choose k-i}D_b^{\ell-k+i}D^{k-i}K_k.
$$
In view of the formula (\ref{eq:rec_K_k}), the previous line can be re-written as
$$
K_i^b=\sum_{\ell=k-i}^k\frac{1}{\ell!}{\ell \choose k-i}(k-i)!D_b^{\ell-k+i}K_i
$$
which is exactly item (1).

Item (2) can be deduced by combining item (1) and Lemma \ref{lem:rec_int}.
\end{proof}


\section*{Acknowledgments}
The work of C.\,A.\,Evripidou was partially supported by a postdoctoral fellowship of the University of Cyprus
and by the Australian Research Council.

\bibliographystyle{abbrv}

\end{document}